\documentclass[11pt]{amsart}

\usepackage{amssymb}
\usepackage{amsfonts}
\usepackage{amsmath}
\usepackage{amsthm}

\usepackage[english]{babel}
\usepackage[latin1]{inputenc}
\usepackage{amsmath,amssymb,amsthm,epsfig}
\usepackage{eufrak}
\usepackage{color}
\usepackage{mathrsfs}
\usepackage{latexsym}

\newcommand{\F}{{\mathcal{F}}}
\newcommand{\G}{{\mathcal{G}}}

\renewcommand{\O}{{\mathcal O}}

\renewcommand{\dim}{\mathrm{dim}\,}

\newcommand{\umapright}[1]{\hskip3pt\smash{\mathop{\longrightarrow}\limits^{#1}}\hskip3pt}

\theoremstyle{plain}

\newtheorem{teo}{Theorem}[section]

\newtheorem{prop}{Proposition}[section]
\newtheorem{defi}{Definition}[section]
\newtheorem{lema}{Lemma}[section]

\theoremstyle{definition}

\newtheorem{obs}{Remark}

\input{xypic.tex}
\xyoption{all}
\usepackage{ifthen,latexsym}
\usepackage{times, amscd}

\begin{document}
\title{hypersurfaces invariant by pfaff systems}
\author{Maurício  Corrêa  Jr.,  Luis G. Maza and Márcio G. Soares }
\address{ Maurício  Corrêa  Jr., Dep. Matemática
ICEx - UFMG, Campus Pampulha,
31270-901 Belo Horizonte - Brasil.}
\email{mauricio@mat.ufmg.br}

\address{Luis G. Maza ,
Departamento de Matematica - UFAL, Campus A.C. Sim\~oes s/n, 57072-090
Maceió - Brasil.}
\email{lmaza@im.ufal.br}

\address{Márcio G. Soares,
Dep. Matemática
ICEx - UFMG, Campus Pampulha
31270-901 Belo Horizonte - Brasil.
}\email{msoares@mat.ufmg.br}

\subjclass{32S65, 58A17. Partially supported by CNPq, CAPES and FAPEMIG, Brasil}

\begin{abstract}
We present results expressing conditions for the existence of meromorphic first integrals for Pfaff
systems of arbitrary codimension on complex manifolds. Some of the results presented improve
previous ones due to J-P. Jouanolou and E. Ghys. We also present an enumerative result counting the number of  hypersurfaces
invariant by a projective holomorphic foliation with split tangent sheaf.
\end{abstract}

\maketitle

\section{INTRODUCTION}

The theory of algebraic integrability introduced by G. Darboux was pursued and expanded by several people, among which we would like to quote H. Poincar\'e, P. Painlev\'e, L. Autonne and, more recently, J-P. Jouanolou and E. Ghys.

In \cite{J}, J-P. Jouanolou gave an extensive account of the theory of algebraic Pfaff equations which included several results on algebraic integrability, among which conditions for the
existence of rational first integrals for Pfaff equations on the projective spaces $\mathbb{P}_{\mathbb{K}}^n$, where $\mathbb{K}$ is an algebraically closed field of characteristic zero.

This quest was continued in \cite{J2}
where he proved that, on a connected compact complex manifold $M$ satisfying certain geometrical conditions, the maximum number of irreducible hypersurfaces which are solutions of a codimension one Pfaff equation
$\omega =0$, $\omega \in H^0(M, \Omega^1_M \otimes_{\mathscr{O}_M} \mathscr{L})$, $\mathscr{L}$ a line bundle on $M$, is bounded by
\begin{equation}\label{jo}
\kappa_J =\dim_{\mathbb{C}} \left[H^0(M, \Omega^2_M \otimes_{\mathscr{O}_M} \mathscr{L})\,/\, \omega \wedge H^0(M, \Omega^1_M) \right] + \varrho +1
\end{equation}
where $\varrho$ is the Picard number of $M$. Further, if the number of invariant hypersurfaces is at least $\kappa_J$, then there are infinitely many such hypersurfaces and this happens if, and only if, the
Pfaff equation $\omega =0$ admits a meromorphic first integral.

This result was improved by E. Ghys in \cite{G}. Ghys showed that
Jouanolou's result holds if the geometrical conditions on $M$,
assumed in \cite{J2}, are dropped.  More precisely, the result of
\cite{G} states that if a Pfaff equation $\omega=0$, given by
$\omega \in H^0(M, \Omega^1_M \otimes_{\mathscr{O}_M} \mathscr{L})$,
does not admit a meromorphic first integral, then the number of
invariant irreducible divisors must be smaller than
\begin{equation}\label{ghys}
\kappa_G = \mathrm{dim}_{\mathbb{C}}\left[\mathrm{H}^0(M,\Omega_{M}^2\otimes_{\mathscr{O}_M}  \mathscr{L})/\omega\wedge\mathrm{H}^0(M,\Omega_{cl}^1)\right] + \mathrm{dim}_{\mathbb{C}} \mathrm{H}^1(M,\Omega_{cl}^1) + 2,
\end{equation}
where $\Omega_{cl}^1$ denotes the sheaf of closed holomorphic
$1$-forms on $M$.

In the case of one-dimensional Pfaff systems, that is, in the case of holomorphic foliations of dimension one, M. Corr\^ea Jr. \cite{MBCJr} presented a similar result which reads:
let $M$ be as above and let $\F$ be a one-dimensional holomorphic foliation on $M$. If $\F$ admits at least
\begin{equation}\label{Cjr}
\kappa_{CJ} =
\dim_{\mathbb{C}}\left[\mathrm{H}^0(M,K_{\F})/i_{X_{\F}}(\mathrm{H}^0(M,\Omega_{cl}^1))\right]+\dim_{\mathbb{C}} \mathrm{H}^1(M,\Omega^1_{cl})+\dim_{\mathbb{C}}M
\end{equation}
invariant irreducible hypersurfaces, then $\F$ admits a meromorphic first integral. Here, $K_{\F}$ is the canonical  bundle of $\F$ and $i_{X_{\F}}(\cdot)$
is the contraction by a vector field $X_{\F}$ inducing the foliation $\F$.

Another result in this direction was presented by S. Cantat in \cite{Can}. Its has a dynamical flavour and reads: let $M$ be as before and let $f$ be a
 dominant endomorphism of $M$. If there are $\kappa_C$
totally $f$-invariant hypersurfaces $W_i \subset M$ (a hypersurface $W \subset M$ is totally $f$-invariant provided $f^{-1}(W) = W$) with
\begin{equation}
\kappa_C > \dim_{\mathbb{C}} H^1(M, \Omega^1_M) + \dim_{\mathbb{C}} M,
\end{equation}
then there is a non constant meromorphic function $\Phi$ and a non zero complex number $\alpha$ such that $\Phi\circ f = \alpha \,\Phi$.

In this note we will consider higher codimensional Pfaff systems on connected complex manifolds.

The affine case was dealt with in \cite{JPAA} and the result is: let $\mathbb{K}$ be an algebraically closed field of characteristic zero and $\omega$ a polynomial $r$-form of degree $d$ on $\mathbb{K}^n$. If $\omega$ admits
\begin{equation}\label{jpaa}
{d-1+n \choose n}\cdot{n \choose r+1}+r+1
\end{equation}
invariant irreducible algebraic hypersurfaces, then $\omega$ admits a rational first integral.

We now state the results of this note, the first one being a higher codimensional counterpart of Jouanolou and Ghys theorems.

\begin{teo}\label{Ghys}
Let $\F$ be a codimension $r$  Pfaff system on a connected, compact,  complex
manifold $M$,  defined by  $\omega\in H^{0}(M,
\Omega^{r}_M \otimes \mathscr{L})$.  If $\F$ admits
\begin{equation}\label{G0}
\dim_{\mathbb{C}}\left[H^{0}(M,
\Omega^{r+1}_M \otimes \mathscr{L}) / \omega \wedge H^{0}(M,
\Omega_{cl}^1)\right] + \dim_{\mathbb{C}}H^{1}(M, \Omega_{cl}^{1}) + r + 1
\end{equation}
invariant irreducible analytic hypersurfaces, then $\F$ admits a
meromorphic first integral.
\end{teo}

The second result makes use of  extactic varieties which are defined in terms of jet bundles and determinants and are discussed in Section \ref{secaoext}.

\begin{teo}\label{integral2}  Let $\F$ be a  holomorphic foliation of dimension $r$, with locally free tangent sheaf, on a complex manifold $M$ and
let $V$ be a finite dimensional linear system on $M$ with
$\dim_\mathbb{C} V = k$. If  $\mathrm{rk}\,T^{(k)}< k$,  then $\F$
has a meromorphic first integral, where
$$ T^{(k)}: V \otimes \mathscr{O}_{M}\longrightarrow J_{\F}^{k}V.$$
Moreover,  if $M$ is compact, then there
exists an algebraic manifold $N$, of dimension $m=m(k,r)$, and a
meromorphic map
$$\varphi:M\dashrightarrow N,$$ such that the fibers of $\varphi$ are $\F$-invariant.
\end{teo}

The third result is enumerative and projective, in which case we can use the notion of degree of a projective Pfaff equation, as defined in Section \ref{enum}. It states

\begin{prop}\label{counting}
Let $\F$ be a holomorphic foliation on
$\mathbb{P}^n$ of dimension $r$, of degree $d>0$ and with split tangent sheaf, that is, $T\F=\bigoplus_{i=1}^r T\F_i =\bigoplus_{i=1}^r\mathscr{O}_{\mathbb{P}^n}(1-d_i)$.  Suppose no $\F_i$ has a rational first integral, $i=1, \dots, r$. Then, the number of $\F$-invariant irreducible
hypersurfaces of degree $\nu$, counting multiplicities, is bounded by
\begin{equation}\label{numinv}
{\nu+n\choose n}+\frac{\deg(\F)}{\nu r}{{\nu + n\choose n}\choose 2} - \frac{1}{\nu} {{\nu + n\choose n}\choose 2} .
\end{equation}
\end{prop}

\section{PRELIMINARIES}

Let $M$ be a connected complex manifold of dimension $n$. Following Jouanolou \cite{J2} we have:

\begin{defi}\label{P}
Let $\mathscr{L}$ be a holomorphic line bundle on $M$. A Pfaff system of codimension $r$ on $M$, $1 \leq r \leq \dim_\mathbb{C}M-1$, is a nontrivial section $\omega \in H^0(M, \Omega^r \otimes_{\mathscr{O}_M} \mathscr{L})$, where $\Omega^r$ is the sheaf of holomorphic $r$-forms on $M$.
\end{defi}
We shall assume that the zero set of $\omega$ has codimension $\geq 2$ (this is not at all restrictive, as explained in \cite{BM}).

Now, given a hypersurface $X \subset M$ consider the inclusion $\textit{i}: X \to M$. We let $\textit{i}_X \omega$ denote the section $\textit{i}_X \omega \in H^0(\Omega_X^r \otimes \mathscr{L}_{|X})$ obtained by the projection of $\omega$ via $(\Omega^r \otimes \mathscr{L})_{|X} \longrightarrow \Omega_X^r \otimes \mathscr{L}_{|X}$.

\begin{defi}\label{Pinv}
$X$ is a solution of $\omega$ if $\textit{i}_X \omega \equiv 0$.
\end{defi}

By a slight abuse of language we will say that $X$ is $\omega$-\emph{invariant} if $X$ is a solution of $\omega$.

Now for foliations.
\begin{defi}\label{fol}
Let $M$ be a complex manifold of dimension and $ \mathscr{O}(TM)$ be its tangent sheaf. A singular holomorphic foliation $\F$ on $M$, of dimension $r$, is a coherent involutive subsheaf $T\F$ of $\mathscr{O}(TM)$ of rank $r$. Involutive (or integrable) means that, for each $x \in M$, the stalk $T\F_x$ is closed under the Lie bracket operation, $[T\F_x, T\F_x\,]\subset T\F_x$.
\end{defi}
In the above, the rank of $T\F$ is the rank of its locally free part. Since $\mathscr{O}(TM)$ is locally free, the coherence of $T\F$ simply means that it is locally finitely generated. We call
$T\F$ the {\it tangent sheaf} of the distribution and the quotient,
$\mathcal{N}_{\F}=\mathscr{O}(TM)/T\F$, its {\it normal sheaf}.

The \emph{singular set} of $\F$ is defined by
$$
S(\F) = \{ x \in M \,:\, (\mathcal{N}_{\F})_x \; \mathrm{is\;
not\; a\; free \;} \mathscr{O}_x-\mathrm{module}\}.
$$
On $M \setminus S(\F)$ there is a unique (up to isomorphism) holomorphic vector subbundle $E$ of
the restriction ${TM}_{\vert\,{M \setminus S(\F)}}$, whose sheaf of
germs of holomorphic sections, $\widetilde{E}$, satisfies
$\widetilde{E}= T\F_{\vert\,{M \setminus S(\F)}}$. Clearly $r=$ rank of $E$.

We will assume that $T\F$ is {\sl full} (or saturated) which
means: let $U$ be an open subset of $M$ and $\xi$ a holomorphic section of $\mathscr{O}(TM)_{|U}$ such that $\xi_x \in T\F_x$ for all $x \in U \cap (M \setminus S(\F))$. Then we have that for all $x \in U$,
$\xi_x \in T\F_x$. In this case the foliation $\F$ is said to be {\sl reduced}.

An equivalent formulation of \textsl{full} is as follows: let $ \mathscr{O}(T^\ast M)$ be the cotangent sheaf of $M$. Set $T\F^o = \{ \omega \in
\mathscr{O}(T^\ast M)\,:\, i_{\gamma}\omega =0\,\,\forall \;\gamma \in T\F \}$ and
$T\F^{oo} = \{ \gamma \in \mathscr{O}(TM)\,:\, i_{\gamma}\omega
=0\,\,\forall \;\omega  \in T\F^o\}$, where $i$ is the contraction. $T\F$ is full if
$T\F = T\F^{oo}$. Note that integrability of
$T\F$ implies integrability of $T\F^{oo}$.

Singular  foliations can dually be defined in terms of the cotangent
sheaf. Thus a {\sl singular foliation of corank $q$}, $\G$, is a
coherent subsheaf $N\G$ of rank $q$ of $\mathscr{O}(T^\ast M)$,
satisfying the integrability condition
$$
d N\G_x \subset (\mathscr{O}(T^\ast M) \wedge N\G)_x
$$
outside the set $S(\G) = \mathrm{Sing}(\mathscr{O}(T^\ast M) /
N\G)$.

 $N\G$ is called the {\sl conormal} sheaf of the foliation $\G$. Its annihilator
\[
N\G^o=\{\, \gamma \in \mathscr{O}(TM)\,:\, i_\gamma \omega=0\ \
\text{for all} \ \omega \in N\G\,\}
\]
 is a singular foliation of dimension $r=\dim M-q$.  See T. Suwa \cite{Suw} for the relation between these two definitions.

We remark that, if a foliation $\F$ is reduced then
$\mathrm{codim}\, S(\F) \geq 2$ and reciprocally, provided $T\F$ is
locally free (see \cite{Suw}). This is a useful concept since it
avoids the appearance of ``fake" (or ``removable") singularities.

Also, a singular foliation $\G$ of corank $q$  induces  a Pfaff
system of codimension $q$. In fact, The $q$-th wedge product of the
inclusion $N\G \subset \mathscr{O}(T^\ast M)$ gives rise to a
nonzero twisted differential $q$-form $\omega_{G}\in
H^0(\Omega^k_X\otimes\det N\G)$.

A last definition. For the concept of meromorphic map refer to
\cite{R}.
\begin{defi}\label{intpri}
If $\F$ is a Pfaff equation or a singular holomorphic foliation on $M$, a first integral for $\F$ is a non-constant meromorphic map $f : M \longrightarrow N$, where $N$ is a complex manifold, such that the fibers of $f$ are $\F$-invariant.
\end{defi}

\section{PROOF OF THEOREM \ref{Ghys}}

Let $M$ be a complex manifold. Denote $\Omega_{cl}^1$ the sheaf of germs of closed holomorphic differential $1$-forms on $M$. We recall the statement:\\

\noindent\textbf{Theorem 1.1}\emph{
Let $\F$ be a Pfaff system on a compact, connected, complex
manifold $M$,  induced by an $r$-form $\omega\in H^{0}(M,
\Omega^{r} \otimes \mathscr{L})$.  If $\F$ admits
\begin{equation}\label{G0}
\dim_{\mathbb{C}}H^{1}(M, \Omega_{cl}^{1}) + \dim_{\mathbb{C}}\left(H^{0}(M,
\Omega^{r+1} \otimes \mathscr{L}) / \omega \wedge H^{0}(M,
\Omega_{cl}^1)\right) + r + 1
\end{equation}
invariant irreducible analytic hypersurfaces, then $\F$ admits a
meromorphic first integral.
}

\begin{proof}
Note that, by the compactness of $M$, the spaces appearing in (\ref{G0}) are finite dimensional. Part of the proof recalls the arguments of \cite{G}. Denote by $\mathrm{Div}(M,\F)$ the abelian group of divisors on $M$ which are
invariant by $\F$. We have the homomorphism
\begin{equation}\label{G1}
\begin{array}{c}
 \mathrm{Div}(M,\F) \longrightarrow \mathrm{Pic}(M) \\
  \\
\hskip 60pt \sum\limits_\alpha\, \lambda^\alpha\, L^\alpha \longmapsto \bigotimes\limits_\alpha\, [L^\alpha]^{\otimes \lambda^\alpha}, \lambda^\alpha \in \mathbb{Z}.
\end{array}
\end{equation}
Since $\mathrm{Pic}(M) \simeq H^1(M, \O^\ast)$, logarithmic differentiation defines a homomorphism
\begin{equation}\label{G2}
\begin{array}{c}
H^1(M, \O^\ast) \longrightarrow H^{1}(M, \Omega_{cl}^1) \\
  \\
\hskip 34pt g \longmapsto \displaystyle\frac{d g}{g}.
\end{array}
\end{equation}
Composition of (\ref{G1}) and (\ref{G2}) gives a
$\mathbb{C}$-linear map
\begin{equation}\label{G3}
\Psi: \mathrm{Div}(M,\F) \otimes \mathbb{C} \longrightarrow H^{1}(M, \Omega_{cl}^1)
\end{equation}
which is expressed, in terms of a sufficiently fine open cover $\{U_i\}_{i \in \Lambda}$ of $M$ by: if $L^\alpha$ is defined by $f_i^\alpha =0$ in $U_i$ and, in $U_{i j}= U_i \cap U_j$, $f_i^\alpha = g_{i j}^\alpha \, f_j^\alpha$,
\begin{equation}\label{G4}
\Psi\left(\sum\limits_\alpha\, \lambda^\alpha \, L^\alpha\right)  = \left[\sum\limits_\alpha\, \lambda^\alpha \, \frac{dg_{i j}^\alpha}{g_{i j}^\alpha}\right].
\end{equation}

Consider the kernel of $\Psi$. $\sum\limits_\alpha\, \lambda^\alpha \, L^\alpha \in \ker \Psi$ amounts to saying that, on each $U_i$, there are closed holomorphic 1-forms $\varpi_i$ such that in $U_{i j}$,
\begin{equation}\label{G5}
\sum\limits_\alpha\, \lambda^\alpha \, \frac{dg_{i j}^\alpha}{g_{i j}^\alpha}=
\varpi_j - \varpi_i.
\end{equation}
But this says that
\begin{equation}\label{G6}
\sum\limits_\alpha\, \lambda^\alpha \, \frac{df_{i}^\alpha}{f_{i}^\alpha} + \varpi_i = \sum\limits_\alpha\, \lambda^\alpha \, \frac{df_{j}^\alpha}{f_{j}^\alpha} + \varpi_j.
\end{equation}
Hence, these glue together to give a global closed meromorphic 1-form $\eta$ on $M$, defined up to summation of a global closed holomorphic 1-form $\rho$.

Remark that, since $L^\alpha$ is $\omega$-invariant and hence $(\omega \wedge  df_{i}^\alpha)_{\mid(f_{i}^\alpha=0)} \equiv 0$,  $\omega \wedge \eta$ is a holomorphic $r+1$-form, defined up to summation of $\omega \wedge \rho$, with $\rho$ a global closed holomorphic 1-form. It follows that the $\mathbb{C}$-linear map
\begin{equation}\label{G7}
\begin{array}{ccc}
\Theta : \ker \Psi & \longrightarrow & \mathrm{H}^{0}(M, \Omega^{r+1} \otimes \mathscr{L})/ \omega \wedge \mathrm{H}^{0}(M, \Omega_{cl}^1) \\
 \\
\sum\limits_\alpha\, \lambda^\alpha \,L^\alpha & \longmapsto & \overline{\omega \wedge \left(\eta + \rho\right)} \hskip 100pt
\end{array}
\end{equation}
is well-defined.

Consider the exact sequence

\begin{equation}\label{G8}
\xymatrix{ 0\ar[r]&\Omega^1 \ar[r] &
\mathscr{M}^1\ar[r]&\mathscr{Q}_{M}^1 \ar[r]&0 }
\end{equation}
where $\mathscr{M}^1$ is the sheaf of germs of meromorphic 1-forms on $M$ and $\mathscr{Q}_{M}^1$ is the quotient $\mathscr{M}^1/\Omega^1$. The following is a version of \cite[Lemme 3.1.1, p. 102]{J}; we repeat the proof here for the sake of completness.

\begin{lema}\label{jou} The $\mathbb{C}$-linear map
\begin{equation}\label{G9}
\begin{array}{ccc}
\mathrm{Div}(M, \F)\otimes \mathbb{C} & \longrightarrow & \mathscr{Q}_{M}^1  \\
  \\
\displaystyle\sum_{\alpha}\lambda^\alpha \cdot L^\alpha & \longmapsto & \displaystyle\sum_{\alpha}\lambda^\alpha\frac{df^{\alpha}_i}{f^{\alpha}_i}
\end{array}
\end{equation}
is injective provided the divisors have no common factor.
\end{lema}
\begin{proof}
Let $\lambda^{1}, \ldots ,\lambda^{s}\in \mathbb{C}$ and $f^{1}, \ldots
,f^{s}\in \mathrm{Div}(M,\F)$ be such that

\begin{equation}\label{G10}
\eta_i =\sum_{j=1}^{s}\lambda^{j}\frac{df^{j}_i}{f^{j}_i}=0\in \mathscr{Q}_{M}^1
\end{equation}
so that $\eta_i$ is holomorphic. In each $U_i$, set $V_i^j = \{f_i^j =0\}$
and choose a point
$$
p_{i}^j \in \displaystyle V_{i}^j\setminus \bigcup_{1\leq k\leq s
\atop_{j\neq k}}V_{i}^k.
$$
Now, let $\Gamma_{i}^j$ be a curve through $p_i^j$ not contained in
$V_{i}^j$ and $\gamma_i^j:\{z;\ |z|<1\}\rightarrow U_{i}$ be a parametrization of $\Gamma_{i}^j$
with $\gamma_i^j(0)=p_{i}^j$. Put $\displaystyle h_{i}^j=f_{i}^j\circ \gamma_i^j$. For all
$k=1, \ldots ,s,\, k\neq j$, the functions $\displaystyle\frac{({h_{i}^{k}})'}{h_{i}^k}$ are regular at $0$, $\displaystyle\frac{({h_{i}^{j}})'}{h_{i}^j}$ has a
simple pole at $0$ and $\mathrm{Res}\left(\displaystyle\frac{({h_{i}^{j}})'}{h_{i}^j},0\right)=\rho_i^j$,
which is   the multiplicity of
$0$ as a zero of $h_{i}^j$. On the other hand, it follows from
$(\ref{G10})$ that
\begin{equation}\label{G11}
\lambda_i^{1}\frac{({h_{i}^{1}})'}{h_{i}^1}+ \cdots +
\lambda_i^{j}\frac{({h_{i}^{j}})'}{h_{i}^j}+ \cdots+
\lambda_i^{s}\frac{({h_{i}^{s}})'}{h_{i}^s}=(\gamma_i^{j})^\ast\eta.
\end{equation}
Taking residues at  $0$ in $(\ref{G11})$ we get  $\lambda_i^{j}\rho_i^j=\mathrm{Res}\left((\gamma_i^j)^*\eta,0\right)=0$ and, as
 $\rho_i^j\neq0$, we have $\lambda_i^{j}=0$ for $j=1,\dots,s$.
\end{proof}

Suppose now that   $\F$ admits at least
\begin{equation}\label{G12}
\dim_{\mathbb{C}}H^{1}(M, \Omega_{cl}^1) + \dim_{\mathbb{C}}\left(H^{0}(M,
\Omega^{r+1} \otimes \mathscr{L}) / \omega \wedge H^{0}(M,
\Omega_{cl}^1)\right) + r + 1
\end{equation}
invariant irreducible analytic hypersurfaces.
From (\ref{G3}) and (\ref{G7}) we have
\begin{equation}\label{G13}
\dim_{\mathbb{C}}\ker \Psi \geq \dim_{\mathbb{C}}(\mathrm{H}^{0}(M,
\Omega^{r+1} \otimes \mathscr{L}) / \omega \wedge H^{0}(M,
\Omega_{cl}^1)) + r + 1.
\end{equation}
Hence
$\dim_{\mathbb{C}} \ker \Theta \geq r+1$ and $\ker \Theta$ is non trivial. Now, given a non zero element
 $x \in \ker \Theta$,
we can choose $L^{1}, \ldots ,L^{k}\in \mathrm{Div}(M, \F)$ and  $\lambda^{1},
\ldots , \lambda^{k} \in \mathbb{C}^{\ast}$ such that
\begin{itemize}
\item [$i)$]$x=
\displaystyle\sum_{\alpha = 1}^{k} \lambda^{\alpha} \, L^{\alpha} \in
\ker \Psi$ and $x$ is nonzero.
  \\
\item [$ii)$]$\displaystyle \Theta \left(\sum_{\alpha = 1}^{k} \lambda^{\alpha}
\, L^{\alpha}\right) = \overline{\,0} \in  H^{0}(M, \Omega^{r+1} \otimes
\mathscr{L}) / \omega \wedge H^{0}(M, \Omega_{cl}^1).$
\end{itemize}

Using (\ref{G6}) we have that there exists  $\mu =
(\mu_i) \in \mathcal{Z}^{0}(M, \Omega_{cl}^1)$ such that, in $U_i$,
\begin{equation}\label{G14}
\omega \wedge \left(\varpi_i + \sum_{\alpha = 1}^{k}
\lambda_{\alpha}\frac{df_{i}^{\alpha}}{f_{i}^{\alpha}}\right) = \omega
\wedge \mu_{i},
\end{equation}
which amounts to
\begin{equation}\label{G15}
\omega \wedge \left(\varpi_{i} - \mu_{i}+ \sum_{\alpha =
1}^{k} \lambda_{\alpha}\frac{df_{i}^{\alpha}}{f_{i}^{\alpha}}\right) = 0
\end{equation}
in each $U_i$. Thus we get a global closed meromorphic
$1$-form $\widetilde{\xi}$ with
\begin{equation}\label{G16}
\widetilde{\xi}_{|U_i}=\varpi_{i} - \mu_{i}+ \sum_{\alpha =
1}^{k} \lambda_{\alpha}\frac{df_{i}^{\alpha}}{f_{i}^{\alpha}}
\end{equation}
such that
\begin{equation}\label{G17}
\omega \wedge
\widetilde{\xi} = 0.
\end{equation}
In this way we can construct   $r+1>1$ global closed meromorphic $1$-forms $\widetilde{\xi}_{1}, \ldots
,\widetilde{\xi}_{r+1}$ on  $M$ such that
\begin{equation}\label{G18}
\omega \wedge
\widetilde{\xi}_{j} = 0\,, \qquad 1 \leq j \leq r+1.
\end{equation}
Define $\alpha_{1}=\widetilde{\xi}_{1} \wedge \cdots \wedge
\widetilde{\xi}_{r}$ and $\alpha_{2}=\widetilde{\xi}_{2} \wedge \cdots \wedge \widetilde{\xi}_{r+1}$ and remark that $|\alpha_{1}|_{\infty}\neq |\alpha_{2}|_{\infty}$, where $|\cdot|_{\infty}$ denotes the set of poles.
We have the following possibilities:
\\
\begin{enumerate}
  \item []\textbf{ Case 1}. $\alpha_{1} \neq 0$ and
$\alpha_{2} \neq 0.$\\
  \item [] \ \textbf{Case 2}. $\alpha_{1}=0$ or
$\alpha_{2}=0.$
\\
\end{enumerate}

\noindent \textbf{Case 1:} Since
$|\alpha_{1}|_{\infty}\neq |\alpha_{2}|_{\infty}$
we have  that
$\alpha_{1}$ and $\alpha_{2}$  are linearly independent
 over  $\mathbb{C}$. We claim there exist $R_1,\,R_2\in \mathscr{M}(M)$
such that  $\omega = R_{i}\,\alpha_{i}$, $i=1,2$. In fact, consider the $\mathscr{M}(M)$-basis $\{\widetilde{\xi}_1, \dots,  \widetilde{\xi}_{r+1}, \dots,\widetilde{\xi}_s\}$ of $\mathscr{M}^1$ obtained by completing
$\{\widetilde{\xi}_1, \dots, \widetilde{\xi}_{r+1}\}$ to a $\mathscr{M}(M)$-basis. Write
\begin{equation}\label{G19}
\omega=\sum_{1\leq i_1<\dots<i_r\leq s} R_{i_1,\dots,i_r}\;
\widetilde{\xi}_{i_1}\wedge\cdots\wedge \widetilde{\xi}_{i_r}.
\end{equation}
Since $\omega \wedge \widetilde{\xi}_j=0$, for $j=1,\dots,r+1$, we have $\omega= R_{1,\dots,r}\;
\widetilde{\xi}_{1}\wedge\cdots \wedge\widetilde{\xi}_{r}$ and $\omega= R_{2,\dots,r+1}\;
\widetilde{\xi}_{2}\wedge\cdots \wedge\widetilde{\xi}_{r+1}$. Note $R_1 = R_{1,\dots,r}$ and $R_2 = R_{2,\dots,r+1}$.
Hence,
$\alpha_{1}=\displaystyle\frac{R_{1}}{R_{2}}\,\alpha_{2}$. Since the meromorphic
$r$-forms $\alpha_{1}$ and $\alpha_{2}$ are closed we have
$0=d\left(\displaystyle\frac{R_{1}}{R_{2}}\right) \wedge \alpha_{2}$ and thus
$d\left(\displaystyle\frac{R_{1}}{R_{2}}\right) \wedge \omega=0$. Moreover, since
$\alpha_{1},\alpha_{2}$  are linearly independent  over  $\mathbb{C}$, $R=\displaystyle\frac{R_{1}}{R_{2}}$ is not constant.
This shows that the meromorphic function  $R$ is a first integral for $\omega$.\\

\noindent \textbf{Case 2:} Suppose  $\alpha_{1}=0$. Let  $m$ be the
largest integer such that $\widetilde{\xi}_{1}, \ldots ,\widetilde{\xi}_{m}$ are
linearly independent over  $\mathscr{M}(M)$. Then,
\begin{equation}\label{G20}
\widetilde{\xi}_{m+1}=\sum_{i=1}^{m}R_{i} \;
\widetilde{\xi}_{i}
\end{equation}
with $R_{1},\dots, R_m \in \mathscr{M}(M)$. Since $\widetilde{\xi}_{i}$ is  closed for  $i=1,\dots,m+1$, we
get
\begin{equation}\label{G21}
0=\sum_{i=1}^{m}dR_{i} \wedge \widetilde{\xi}_{i}.
\end{equation}
Then, for each $j=1,\dots,m$, multiplying (\ref{G21}) by $ \widetilde{\xi}_{1}
\wedge \cdots \wedge \widehat{\widetilde{\xi}_{j} }\wedge \cdots \wedge
\widetilde{\xi}_{m}$ we obtain

\begin{eqnarray*}
0&=&\sum_{i=1}^{m}dR_{i} \wedge \widetilde{\xi}_{i} \wedge \widetilde{\xi}_{1} \wedge
\cdots \wedge \widehat{\widetilde{\xi}_{j} }\wedge \cdots \wedge
\widetilde{\xi}_{m}\\
&=&(-1)^{j+1}dR_{j} \wedge \widetilde{\xi}_{1} \wedge \cdots \wedge \widetilde{\xi}_{m}.
\end{eqnarray*}\\

Since  $\widetilde{\xi}_{1}, \cdots ,\widetilde{\xi}_{m}$  are linearly independent
over  $\mathscr{M}(M)$, there exist $g_{1},
\cdots ,g_{m}$ $\in  \mathscr{M}(M)$ such that  $dR_{j}=\sum_{i=1}^{m}g_{i}
\; \widetilde{\xi}_{i}.$ By $(\ref{G18})$ we get
$$
dR_{j} \wedge \omega=\sum_{l=i}^{m}g_{i}
\; \widetilde{\xi}_{i} \wedge \omega=0.
$$
Besides, from $(\ref{G20})$ and lemma \ref{jou}
there exists $i_{0}\in \{1,\dots,m\}$ such that
$R_{i_{0}}$ is not constant. That is, $R_{i_{0}}$ is a meromorphic
first integral for  $\omega$. The case $\alpha_{2}=0$ is dealt with analogously.

\end{proof}

\section{JET BUNDLES AND EXTACTIC VARIETIES}\label{secaoext}

Throughout this section $\F$ denotes a (singular) holomorphic foliation of dimension $r$, on a connected complex manifold of dimension $n$. We assume $\F$ has a locally free tangent sheaf and that its singular locus has codimension greater than or equal to $2$. Hence, such an $\F$ is given by the data:
\begin{itemize}
  \item [i)] An open covering $\{U_{\alpha}\}_{\alpha\in \Lambda}$ of
  $M$.
  \\
  \item [ii)] An involutive system of holomorphic vector fields $\{X_{1,\alpha},\dots,X_{r, \alpha}\}$ on $U_{\alpha}$ such that
\begin{equation}\label{PI1}
  S_{\alpha}=\{p\in U_{\alpha}\; ; \; \left(X_{1,\alpha}\wedge\cdots \wedge
  X_{r,\alpha}\right) (p)=0\}
\end{equation}
  satisfies the condition $\mathrm{codim}_{\mathbb{C}} S_{\alpha} \geq 2.$
  \\
  \item [iii)] Whenever $U_{\alpha \beta }:=U_{\alpha  } \cap U_{ \beta }\neq \emptyset$,  there exists $g_{\alpha\beta}:=[g_{\alpha\beta}^{ij}]\in GL(r,\mathscr{O}(U_{\alpha \beta}))$
  such that
\begin{equation}\label{PI2}
\begin{array}{c}
X_{i,\alpha}=\sum\limits_{j=1}^r g_{\alpha\beta}^{ij}\;
X_{j,\beta},\; \mathrm{and\; hence} \\
\\
 X_{1,\alpha}\wedge\cdots \wedge
  X_{r,\alpha}=\det(g_{\alpha \beta}) \;X_{1,\beta}\wedge\cdots \wedge
  X_{r,\beta}\; \mathrm{in}\; U_{\alpha \beta }.
\end{array}
\end{equation}

\end{itemize}

In particular, $S_{\alpha} \cap U_{\alpha \beta }=S_{\beta} \cap
U_{\alpha \beta }$. Thus, the singular set
$\mathrm{Sing}(\F)=\bigcup_{\alpha}S_{\alpha}$ of $\F$ is an
analytic subset of $M$ of $\mathrm{codim}_{\mathbb{C}}
\mathrm{Sing}(\F) \geq 2$. Also, if  $T\F^*$ is the vector bundle
associated to the cocycles
$\left\{[g_{\alpha\beta}^{ij}]:=g_{\alpha\beta}\in
GL(r,\mathscr{O}(U_{\alpha \beta}))\right\}_{\alpha\in \Lambda}$
then, by (\ref{PI2}),  $\F$ induces a global holomorphic section of
$\bigwedge^r TM\otimes \det (T\F^*) $.

\begin{defi}
Let $\vartheta: \mathscr{G}\rightarrow \mathscr{H}$ be a map of coherent sheaves of $\mathscr{O}_M$-modules. Denote by $\vartheta(x): \mathscr{G}(x)\rightarrow \mathscr{H}(x)$ the corresponding linear map on stalks. The rank of $\vartheta$, denoted $\mathrm{rk}(\vartheta)$, is the largest integer $r$ such that $\bigwedge^r \vartheta\neq 0.$ If $x $ is a point in $M$, then the local rank of $\vartheta$  at $x$, $\mathrm{rk}_x(\vartheta)$, is the rank of the map $\vartheta(x)$.
\end{defi}

Before proceeding, a piece of notation. Consider a vector field $X$ as a derivation and, if $f$ is a function, we let $X^0(f)=f$ and
$X^m(f)= X(X^{m-1}(f))$, for $m>0$.

Let $H$ be a holomorphic line
bundle on $M$, consider  a finite dimensional linear system
$V\subset \mathrm{H}^{0}(M,H)$, of dimension $k \geq 1$, and take an
open covering $\{\mathcal{U}_{\alpha}\}_{\alpha\in \Lambda}$ of $M$
which trivializes both $H$ and $(T\F^*)^*=T\F$. In the open set
$\mathcal{U}_{\alpha}$ we can consider the map
$$
T^{(k)}_{\alpha}:V \otimes
\mathscr{O}_{\mathcal{U}_{\alpha}}\rightarrow
\mathscr{O}_{\mathcal{U}_{\alpha}}^{\oplus {k-1+r\choose r}}
$$
defined by
\begin{equation}\label{jet}
T^{(k)}_{\alpha}(s_{\alpha})=\sum_{\begin{array}{c}
i_1+\dots + i_r\leq k
\end{array}} X_{1,\alpha}^{i_1}\dots X_{r,\alpha}^{i_r}(s_{ \alpha})\cdot
\frac{t_1^{i_1}\dots t_r^{i_r}}{i_1!\dots i_r!},
\end{equation}
where $s_{ \alpha}$  and $\{X_{i,\alpha}\}_{i=1}^{r} $ are  local
representatives, respectively, of a section $s\in V \subset
\mathrm{H}^{0}(M,H)$ and of generators of $T\F$. Also, for each $ |I|=i_1+ \dots + i_r$, $0 \leq |I| \leq k$ we fix an order of appearance for the monomials $ t^{I} =t_1^{i_1}\dots t_r^{i_r}$. This is not relevant to the construction of the jet bundles, since different orderings will produce isomorphic bundles and, besides, the extatic variety (Definition \ref{extactic} below) remains unchanged as it is locally given by the vanishing of the $k \times k$ minors of (\ref{tk}).

In
$\mathcal{U}_{\alpha\beta}$ we have $s_{\alpha}=f_{\alpha\beta}
s_{\beta}$ and by (\ref{PI2}) and Leibniz's rule we get
\begin{equation}\label{transition}
\begin{array}{c}
X_{i,\alpha }(s_{\alpha})=\left[\displaystyle\sum_{j=1}^k g_{\alpha\beta}^{ij}\cdot
X_{j,\beta}(f_{\alpha \beta})\right] \cdot s_{\beta}+ f_{\alpha\beta} \cdot\displaystyle\sum_{j=1}^k g_{\alpha\beta}^{ij}\cdot
X_{j,\beta}(s_{\beta}).
\end{array}
\end{equation}
Iterating this procedure up to order ${k-1+r\choose r}$, where $k= \dim_{\mathbb{C}}V$, we
obtain

$$
  \left[
     \begin{array}{c}
       s_{\alpha}\\
       \\
       X_{\alpha}^{J_1}(s_{\alpha}) \\
       \\
     X_{\alpha}^{J_2}(s_{\alpha}) \\
       \\
       \vdots\\
       \\
    X_{\alpha}^{J_k}(s_{\alpha})
     \end{array}
   \right]=
$$

$$
=\left[
\!\!\begin{array}{ccccc}
f_{\alpha\beta} & 0 &\cdots & 0 & 0\\
\\
\left[\sum\limits_{j=1}^k g_{\alpha\beta}^{ij}\cdot
X_{j,\beta}(f_{\alpha \beta})\right]& \!\!\!\!f_{\alpha\beta} \cdot  g_{\alpha\beta}& \cdots
& 0 & 0 \\
         \\
          * & * & f_{\alpha\beta} \cdot g_{\alpha\beta}^2  & 0 & 0
           \\
           \\
         \vdots & \vdots & \vdots & \ddots & \vdots \\
         \\
        * &* & \cdots &
     *& \!\!\!f_{\alpha\beta} \cdot g_{\alpha\beta}^{{k-1+r\choose r}} \\
        \\
       \end{array}
     \!\!\right]\cdot\left[
     \!\!\!\!\begin{array}{c}
       s_{\beta}\\
       \\
      X_{\beta}^{J_1}(s_{\beta}) \\
       \\
      X_{\beta}^{J_2}(s_{\beta}) \\
       \\
       \vdots\\
       \\
        X_{\beta}^{J_k}(s_{\beta})\\
     \end{array}
   \!\!\!\!\right]
$$
where $ X_{\alpha}^{J_{\ell}}(\bullet )=(\cdots, X_{1,\alpha}^{j_1}\cdots X_{r,\alpha}^{j_r}(\bullet ), \cdots )^{\textsf{T}}$, with $|J_\ell|=: j_1+\cdots+j_r=\ell$ for all
$\ell=1,\dots,k.$
Denoting the $\left({k-1+r\choose r}+1\right) \times \left({k-1+r\choose r}+1 \right)$ matrix above by
$\Theta_{\alpha\beta}(\F,V)\in
GL\left({k-1+r\choose r}+1,\mathscr{O}_{\mathcal{U}_{\alpha\beta}}\right)$, a calculation shows that
$$
\left\{
\begin{array}{ll}
\Theta_{\alpha\beta}(\F,V)(p)\cdot\Theta_{\beta\alpha}(\F,V)(p)=I,\ \  $for all$ \ p\in \mathcal{U}_{\alpha \beta}\neq \emptyset\\
\\
\Theta_{\alpha\beta}(\F,V)(p)\cdot\Theta_{\beta\gamma}(\F,V)(p)\cdot\Theta_{\gamma\alpha}(\F,V)(p)=I,\
\ $for all$
\ p\in \mathcal{U}_{\alpha \beta \gamma}\neq\emptyset.\\
\end{array}
\right.
$$
That is, the family  of matrices
$\{\Theta_{\alpha\beta}(\F,V)\}_{\alpha, \beta \in \Lambda}$ defines the cocycles
of a vector bundle of rank ${k-1+r\choose r}+1$ on $M$ that we denote
by $J_{\F}^{k}V$. Now, using the trivializations
$\{\Theta_{\alpha\beta}(\F,V)\}_{\alpha,\beta \in \Lambda}$ we can glue the maps $T^{(k)}_{\alpha}$ and obtain
the morphism
\begin{equation}\label{mtk}
T^{(k)}: V \otimes \mathscr{O}_{M}\rightarrow J_{\F}^{k}V.
\end{equation}

\begin{obs}\label{esteves}
The referee of this article brought to our attention a work by E. Esteves \cite{Est} which was unknown to us,  where a more general construction of bundles and of sheaves of jets of foliations is carried out. In the case of holomorphic foliations the construction given in \cite{Est} coincide with the given by (\ref{mtk}) for $V$ a linear subspace of $H^0(M, \mathscr{O}_M)$. \hfill $\square$
\end{obs}
Taking the $k$-th wedge product  of $T^{(k)}$ and recalling that $\dim_\mathbb{C}V =k$ we have the morphism
\begin{center}
$\bigwedge^k T^{(k)}:\bigwedge^kV\otimes
\mathscr{O}_{M}\rightarrow \bigwedge^k J_{\F}^{k-1}V$,
\end{center}
and tensorizing by $(\bigwedge^kV)^*$ we obtain a global section of
$\bigwedge^kJ_{\F}^{k-1}V\otimes (\bigwedge^kV)^*$ given by
\begin{equation}\label{extk}
\boldsymbol{\varepsilon}(\F,V): \mathscr{O}_{M}\rightarrow
\bigwedge^k J_{\F}^{k-1}V\otimes (\bigwedge^kV)^*.
\end{equation}

\begin{defi}\label{extactic}

The extactic variety of $\F$ with respect to the linear system
$V\subset  \mathrm{H}^{0}(M,H)$ is the locus of zeros
$\mathcal{E}(\F,V)$
of  the section
\begin{center}
$\boldsymbol{\varepsilon}(\F,V) \in
\mathrm{H}^{0}\left(M,\bigwedge^kJ_{\F}^{k-1}V\otimes
(\bigwedge^kV)^*\right).$
\end{center}
The section $\boldsymbol{\varepsilon}(\F,V)$ is the extactic
section of $\F$ with respect to $V$.

\end{defi}

Considering the local expression of $ T^{(k)}$ applied to a basis of $V$, say $\{s_1, \dots, s_k\}$, we have the matrix
\begin{equation}\label{tk}
\left[\begin{array}{cccc}
  s_1^{\alpha}& s_2^{\alpha} & \cdots & s_{k}^{\alpha}\\
  \\
   X_{1,\alpha}(s_1^{\alpha}) &X_{1,\alpha}(s_2^{\alpha}) & \cdots &X_{1,\alpha}(s_{k}^{\alpha})\\
   \\
   X_{2,\alpha}(s_1^{\alpha}) &X_{2,\alpha}(s_2^{\alpha}) & \cdots &
X_{2,\alpha}(s_{k}^{\alpha})
\\
  \vdots & \vdots & \ddots & \vdots\\
  \\
X_{r,\alpha}(s_1^{\alpha}) &X_{r,\alpha}(s_2^{\alpha}) & \cdots &
X_{r,\alpha}(s_{k}^{\alpha})
\\
  \vdots & \vdots & \ddots & \vdots\\
  \\
X_{\alpha}^{J_\ell}(s_1^{\alpha}) &X_{\alpha}^{J_\ell}(s_2^{\alpha}) & \cdots &
X_{\alpha}^{J_\ell}(s_{k}^{\alpha})\\
\vdots & \vdots & \ddots & \vdots\\
  \\
X_{\alpha}^{J_{k-1}}(s_1^{\alpha}) &X_{\alpha}^{J_{k-1}}(s_2^{\alpha}) & \cdots &
X_{\alpha}^{J_{k-1}}(s_{k}^{\alpha})
\end{array}
\right],
\end{equation}
where $ X_{\alpha}^{J_{\ell}}(s_i^{\alpha})=(\cdots, X_{1,\alpha}^{j_1}\cdots X_{r,\alpha}^{j_r}(s_i^{\alpha} ), \cdots )^{\textsf{T}}$, with $|J_\ell|=\ell $  and $s_i^{\alpha}$ is the local representation of the section $s_i$, $i=1,\dots,k$. It follows that the local expression of
the section $\boldsymbol{\varepsilon}(\F,V)$ is given by the determinants of the $k \times k$ minors of (\ref{tk}):
\begin{equation}\label{exsec}
\left( \dots, \det\left[ X_{\alpha}^{J}(s_j^{\alpha})\right], \dots\right)  \quad 0 \leq |J| \leq k-1, \quad 1\leq j\leq k.
\end{equation}
and the local defining equations of $\mathcal{E}(\F,V)$ are
\begin{equation}\label{exeq}
\det\left[ X_{\alpha}^{J}(s_j^{\alpha})\right] =0, \quad 0 \leq |J| \leq k-1, \quad 1\leq j\leq k.
\end{equation}

The simplest situation is that of a single vector field. In this case the extactic $\boldsymbol{\varepsilon}(\F,V)$, where $\F$ is the one-dimensional foliation induced by $X$ is given by
\begin{equation}\label{tkx}
\boldsymbol{\varepsilon}(\F, V) =\det \left[\begin{array}{cccc}
  s_1^{\alpha}& s_2^{\alpha} & \cdots & s_{k}^{\alpha}\\
  \\
   X_{\alpha}(s_1^{\alpha}) &X_{\alpha}(s_2^{\alpha}) & \cdots &X_{\alpha}(s_{k}^{\alpha})\\
   \\
   X^2_{\alpha}(s_1^{\alpha}) &X^2_{\alpha}(s_2^{\alpha}) & \cdots &
X^2_{\alpha}(s_{k}^{\alpha})
\\
\vdots & \vdots & \ddots & \vdots
\\
X^{k-1}_{\alpha}(s_1^{\alpha}) &X^{k-1}_{\alpha}(s_2^{\alpha}) & \cdots &
X^{k-1}_{\alpha}(s_{k}^{\alpha})
\end{array}
\right]
\end{equation}
where $\{ s_1, \dots, s_k\}$ is a basis of $V$. The following results elucidate the role of the extactics in this case.

\begin{prop}\label{jvpp}\cite{JVPF} Let $\F$ be a one-dimensional holomorphic foliation on a complex manifold $M$. If $V$ is a finite dimensional linear system, then every $\F$-invariant divisor in $V$ must be contained in the extactic divisor $\mathcal{E}(\F,V)$.
\end{prop}

\begin{teo}\label{jvpf} \cite[Theorem 3]{JVPF} Let $\F$ be a one-dimensional holomorphic foliation on a complex manifold $M$. If $V$ is a finite dimensional linear system such that $\boldsymbol{\varepsilon}(\F, V)$ vanishes identically, then there exists an open and dense set $U$, possibly intersecting the singular set of $\F$, where $\F_{|U}$ admits a holomorphic first integral. Moreover, if $M$ is a projective variety, then $\F$ admits a meromorphic first integral.
\end{teo}

This last result can be slightly improved as follows:

\begin{teo}\label{mbcjr} \cite[Theorem 1.2]{MBCJr} Let $\F$ be a one-dimensional holomorphic foliation on a complex manifold
 $M$ and let $V$ be a finite dimensional linear system on $M$. If $\boldsymbol{\varepsilon}(\F, V)$ vanishes identically then $\F$ admits a meromorphic first integral.
\end{teo}

We can then make the following

\begin{obs}\label{exfield} Let $\F$ be a one-dimensional holomorphic foliation on a complex manifold $M$ and $V$ a finite dimensional linear system on $M$. If $\F$ does not admit a meromorphic first integral, then $\boldsymbol{\varepsilon}(\F, V) \not\equiv 0$.
\end{obs}

Now we give counterparts of Proposition \ref{jvpp} and Theorems \ref{jvpf} and \ref{mbcjr} for higher dimensional foliations.

\begin{prop}\label{propext} Let $\F$ be a  holomorphic foliation of dimension $r$
on a complex manifold $M$ of dimension $n$. If $V$ is a linear
system of dimension $k$, then every $\F$-invariant complete intersection of  elements  of $V$ is contained in  $\mathcal{E}(\F,V)$.
\end{prop}
\begin{proof}
 Every complete intersection of  elements  of $V$ is the intersection
of the zeros of  $\ell$ $\mathbb{C}$-linearly
independent sections, say  $s_1,\dots, s_\ell \in V$. Then
we can, if necessary, complete this family to a basis of $V$, say
$$\{s_1,\dots,s_\ell, s_{\ell+1},\dots s_{k}\}.$$ Consider in each $\mathcal{U}_{\alpha}$ the representation $s_i^{\alpha}$ of the section $s_i$.
If $Z(s_1^{\alpha},\dots, s_\ell^{\alpha})$ is $\F$-invariant, then
$$X_{i,\alpha}(s_j^{\alpha})\in \mathcal{I}(s_1^{\alpha}, \ldots ,s_{\ell}^{\alpha}),\ \ j=1, \dots, \ell \, ; \, i=1,\ldots,r$$
and so we get $  X_{\alpha}^{J}(s_j^{\alpha})\in \mathcal{I}(s_1^{\alpha}, \ldots ,s_{\ell}^{\alpha})$, for $j=1, \dots, \ell$ and any $J$. Hence, in $\mathcal{U}_{\alpha}$ the determinants
$$
\det\left[ X_{\alpha}^{J}(s_j^{\alpha})\right], \quad 0 \leq |J| \leq k-1, \quad 1\leq j\leq k.
$$
do all vanish along $Z(s_1^{\alpha},\dots, s_\ell^{\alpha})$.
\end{proof}

\section{PROOF OF THEOREM \ref{integral2}}

\noindent\textbf{Theorem 1.2.}\emph{ Let $\F$ be a  holomorphic
foliation of dimension $r$, with locally free tangent sheaf, on a
complex manifold $M$ and  let $V$ be a finite dimensional linear
system on $M$ with $\dim_\mathbb{C} V = k$. If
$\mathrm{rk}\,T^{(k)}< k$,  then $\F$ has a meromorphic first
integral. If $M$ is compact, then there exist an algebraic manifold
$N$, of dimension $m=m(k,r)$, and a meromorphic map
$$\varphi:M\dashrightarrow N,$$ such that the fibers of $\varphi$ are $\F$-invariant.
}

\begin{proof}

Let $\{s_1, \dots, s_k \}$ be a $\mathbb{C}$-basis for $ V$ and suppose
$\mathrm{rank}\, T^{(k)}< k$. This means that the columns of the matrix (\ref{tk}) are dependent over the  field of meromorphic functions $\mathscr{M}(\mathcal{U}_{\alpha})$. Hence
there are meromorphic functions
$\theta_{1}^{\alpha},\dots,\theta_{k}^{\alpha}$ in
$\mathcal{U}_{\alpha}$ such that
\begin{equation}\label{extat}
M_i^{\alpha}=\displaystyle\sum_{j=1}^{k}
\theta_j^{\alpha}X_{i,\alpha}(s_j^{\alpha})=0, \qquad 0\leq i\leq r.
\end{equation}
Now, let $m$ be the smallest integer with the property that
there exist meromorphic  functions
$\theta_1^{\alpha},\dots,\theta_{m}^{\alpha}$ and
$s_1^{\alpha},\dots,s_{m}^{\alpha}\in V$, linearly
independent over $\mathbb{C}$, such that $(\ref{extat})$ holds, that is, $m$ is the rank of $T^{(k)}$ over the field $\mathscr{M}(M)$ of meromorphic functions in $M$. Hence,
\begin{equation}\label{extat1}
\begin{array}{c}
\widetilde{M}_0^{\alpha}= \displaystyle\sum_{j=1}^{m}
\theta_j^{\alpha}s_j^{\alpha}=0, \hfill\\
\,\\
\widetilde{M}_i^{\alpha}=\displaystyle\sum_{j=1}^{m}
\theta_j^{\alpha}X_{i,\alpha}(s_j^{\alpha})=0, \qquad 1\leq i\leq r.
\end{array}
\end{equation}

We have $1 < m \leq k$ and we may assume
$\theta_{m}^{\alpha}=1$. Applying the derivation
$X_{i,\alpha}$ to $\widetilde{M}_0^{\alpha}$ we get, for $1\leq i \leq r$,
\begin{equation}\label{extatica2}
\begin{array}{c}
X_{i,\alpha}(\widetilde{M}_0^{\alpha})= \hfill\\
\, \\
=X_{i,\alpha}(\theta_1^{\alpha})s_1^{\alpha}+\cdots+X_{i,\alpha}(\theta^\alpha_{m-1})s_{m-1}^{\alpha}+
\underbrace{X_{i,\alpha}(\theta^\alpha_{m})}_{=\,0}s_{m}
^{\alpha}+ \widetilde{M}_i^{\alpha}=\\
= X_{i,\alpha}(\theta_1^{\alpha})s_1^{\alpha}+\cdots+
X_{i,\alpha}(\theta^\alpha_{m-1})s_{m-1}
^{\alpha}=0. \hfill
\end{array}
\end{equation}
The minimality of $m$ assures that
$X_{i,\alpha}(\theta_1^{\alpha})=\cdots=
X_{i,\alpha}(\theta_{m-1}^{\alpha})=0$ and hence we have at least one meromorphic first integral common to the
$X_{i,\alpha}$ in $\mathcal{U}_{\alpha}$, $1 \leq i \leq r$, provided
the $\theta_j^{\alpha}$ are not all constants. This is the case since
\begin{equation}\label{uaint}
M_0^{\alpha}=\theta_1^{\alpha}s_1^{\alpha}+\cdots+
\theta_{m-1}^{\alpha}s_{m-1}^{\alpha}+s_{m}^{\alpha}=0
\end{equation}
and $s_1^{\alpha},\dots,s_{m}^{\alpha}\in V$ are linearly
independent over $\mathbb{C}$.

In $\mathcal{U}_{\alpha\beta}$ we have, since $s_j^\beta =f_{\beta\alpha}s_j^{\alpha}$,

\begin{equation}\label{uabint}
\begin{array}{c}
0= \theta_1^{\beta}s_1^{\beta}+\cdots+
\theta_{m-1}^{\beta}s_{m-1}^{\beta}+s_{m}^{\beta}= \hfill\\
  \, \\
=\theta_1^{\beta}f_{\beta\alpha}s_1^{\alpha}+\cdots+
\theta_{m-1}^{\beta}f_{\beta\alpha}s_{m-1}^{\alpha}+f_{\beta\alpha}s_{m}^{\alpha}=\\
\, \\
=f_{\beta\alpha}(\theta_1^{\beta}s_1^{\alpha}+\cdots+
\theta_{m-1}^{\beta}s_{m-1}^{\alpha}+s_{m}^{\alpha}),\hfill
\end{array}
\end{equation}
which amounts to
\begin{equation}\label{ubint}
\theta_1^{\beta}s_1^{\alpha}+\cdots+
\theta_{m-1}^{\beta}s_{m-1}^{\alpha}+s_{m}^{\alpha}=0.
\end{equation}
Taking the difference between (\ref{uaint}) and (\ref{ubint}) we have
\begin{equation}
(\theta_1^{\alpha}-\theta_1^{\beta}) s_1^{\alpha}+\cdots+
(\theta_{m-1}^{\alpha}-\theta_{m-1}^{\beta})s_{m-1}^{\alpha}=0.
\end{equation}
Again, the minimality of $m$ gives
\begin{equation}\label{globint}
\theta_i^{\alpha}-\theta_i^{\beta} =0 \; \mathrm{in}\; \mathcal{U}_{\alpha\beta}\,, \; 1 \leq i \leq m-1.
\end{equation}
Therefore we obtain at least one nonconstant meromorphic
first  integral  $\Theta_i$ given locally by
$\Theta_{i|_{\mathcal{U}_{\alpha}}}=\theta_i^{\alpha}$.

Suppose now that $M$ is compact. Recall that the algebraic dimension $a(M)$ is the transcendence degree over $\mathbb{C}$ of the field $\mathscr{M}(M)$ of meromorphic functions on $M$, that is, $a(M)$ equals the maximal number of elements $\mathsf{f}_1, \dots, \mathsf{f}_a \in \mathscr{M}(M)$ satisfying
$$
d \mathsf{f}_1 \wedge \dots \wedge d \mathsf{f}_a \neq 0.
$$

The algebraic Reduction Theorem \cite[pg. 24, Theorem 3.1]{U} states that there exist a bimeromorphic modification $\widetilde{M} \umapright \zeta M$ and a holomorphic map $\pi : \widetilde{M} \longrightarrow V$, with connected fibres, onto an algebraic manifold $V$ of dimension $a(M)$, such that $\zeta^\ast \mathscr{M}(M) \simeq \pi^\ast \mathscr{M}(V)$. A similar construction holds for any algebraically closed subfield $\mathscr{K} \subset \mathscr{M}(M)$.

Now, let $\mathscr{M}(\F) \subset \mathscr{M}(M)$ be the subfield of meromorphic first integrals of $\F$. Invoking the algebraic Reduction Theorem we conclude, since $\mathscr{M}(\F)$ is a subfield of $\mathscr{M}(M)$, that there exists an algebraic manifold $N$ with
\begin{equation}\label{algred}
\dim N =\mathrm{tr}\!\deg_\mathbb{C} \mathscr{M}(\F)\leq  \mathrm{codim}_{\mathbb{C}} \F
\end{equation}
and  a meromorphic map
$\varphi:M\dashrightarrow N,$ such that the fibers of $\varphi$ are $\F$-invariant.

\end{proof}

\section{AN APPLICATION IN PROJECTIVE SPACES}\label{enum}

The next result is one in projective spaces and in this case we have at hand the notion of degree of a Pfaff system, which we now recall.

\begin{defi}\label{degree}
Let $\F$ be a codimension $n-k$ Pfaff equation on $\mathbb P^n$ given
by $\omega \in H^0(\mathbb P^n, \Omega^{n-k}_{\mathbb P^n} \otimes
\mathcal L)$. If $\mathrm{i}: \mathbb P^{n-k} \to \mathbb P^n$ is a generic
linear immersion then $\mathrm{i}^* \omega \in H^0(\mathbb P^{n-k},
\Omega^{n-k}_{\mathbb P^{n-k}} \otimes \mathcal L)$ is a section of a line
bundle, and its zero divisor reflects the tangencies between
$\F$ and $\mathrm{i}(\mathbb P^{n-k})$. The degree of $\F$ is the degree of such tangency divisor. It is noted $\deg(\F)$.
\end{defi}

Set $d:=\deg(\F)$. Since $\Omega^{n-k}_{\mathbb P^{n-k}}\otimes
\mathcal L= \mathcal O_{\mathbb P^{n-k}}( \deg(\mathcal L) - n+k - 1)$, one concludes that
$\mathcal L= \mathcal O_{\mathbb P^n}(d+ n-k + 1)$. Besides, invoking the Euler sequence a section $\omega$ of
$\Omega^{n-k}_{\mathbb P^n} ( d + n-k + 1  )$ can be
thought off as a polynomial $(n-k)$-form on $\mathbb{C}^{n+1}$ with  homogeneous
coefficients of degree $d + 1$, which we will still
denote by $\omega$, satisfying
\begin{equation}
\label{equirw}
i_\vartheta  \omega = 0
\end{equation}
where $\vartheta=x_0 \frac{\partial}{\partial x_0} + \cdots + x_n
\frac{\partial}{\partial x_n}$ is the radial vector field and $i_\vartheta$ means contraction by $\vartheta$. Thus the
study of Pfaff equations of codimension $n-k$ and degree $d$ on $\mathbb P^n$ reduces to the
study of locally decomposable homogeneous $(n-k)$-forms  on $\mathbb C^{n+1}$, of
degree $d+1$, satisfying relation (\ref{equirw}).

In particular, we have two ways to describe reduced one-dimensional  Pfaff equations $\mathcal{F}$ on $\mathbb{P}^n$, which in fact define foliations since the integrability condition is void in this case. One is through a homogeneous $(n-1)$-form $\omega \in \Omega^{n-1}_{\mathbb{P}^n}(d+n)$ of degree $d+1$ and, using the canonical isomorphism $E \cong \bigwedge^n E \otimes \bigwedge^{n-1} E^\ast$, where $E$ is a vector bundle of rank $n$,  through a section $\chi: T\mathbb{P}^n \otimes \mathscr{O}(d-1)$ that is, a vector field which annihilates $\omega$, $i_\chi \omega =0$. Equivalently, $\mathcal{F}$ is given by a morphism
$$
\Phi: \mathscr{O}(1-d) \longrightarrow T\mathbb{P}^n.
$$
Hence $\mathcal{F}$ is represented, in $\mathbb{C}^{n+1}$, by a homogeneous polynomial vector field $\chi$ of degree $d$ and, due to Euler's sequence, $\chi + g \vartheta$ represents the same foliation, where $g$ is any homogeneous polynomial of degree $d$.\\

\textbf{Proposition 1.1.}\emph{
Let $\F$ be a holomorphic foliation on
$\mathbb{P}^n$ of dimension $r$, of degree $d>0$ and with split tangent sheaf, that is, $T\F=\bigoplus_{i=1}^r T\F_i =\bigoplus_{i=1}^r\mathscr{O}_{\mathbb{P}^n}(1-d_i)$.  Suppose no $\F_i$ has a rational first integral, $i=1, \dots, r$. Then, the number of $\F$-invariant irreducible
hypersurfaces of degree $\nu$, counting multiplicities, is bounded by
\begin{equation}\label{numinv}
{\nu+n\choose n}+\frac{\deg(\F)}{\nu r}{{\nu + n\choose n}\choose 2} - \frac{1}{\nu} {{\nu + n\choose n}\choose 2} .
\end{equation}
}\\

Note that the splitting of the tangent sheaf implies $\deg(\F)= \sum_1^r d_i$.

In particular, if $T\F=\mathscr{O}_{\mathbb{P}^n}^{\oplus r}$, that is to say $\deg(\F_i) =d_i=1$, then the number of $\F$-invariant
hyperplanes is bounded by $n+1$. This bound is sharp in the case of hyperplanes,  as can be seen through the example
\[
 \sum_{i=0}^n \lambda_i \frac{dx_i}{x_i}
\]
where $n \geq 2$, $\sum \lambda_i=0$ and no $\lambda_i$ is equal to zero. This foliation has degree $n-1$ and leaves invariant the $n+1$ hyperplanes $x_i=0$, $i=0, \dots, n$.

This is in accordance with J.V. Pereira and S. Yuzvinsky \cite[Proposition 4.1]{PY} where it is shown that
a foliation $\F$ of codimension one on $\mathbb P^n$ with non-degenerate Gauss map has the number of invariant hyperplanes limited by
\[
\left( \frac{ n+1} { n-1}  \right)  \deg (\mathcal F) .
\]

\begin{proof}
With notations as in Section \ref{secaoext}, let $X_1, \dots, X_r$ be homogeneous polynomial vector fields on $\mathbb{C}^{n+1}$ inducing $\F$, with $X_i$ inducing the foliation $\F_i$, $i=1, \dots, r$. Let $V= H^0(\mathbb{P}^n, \mathscr{O}_{\mathbb{P}^n}(\nu))$ be the linear system consisting of the homogeneous polynomials of degree $\nu$ in $z_0, \dots, z_n$, $\dim V = K= {\nu + n\choose n}$. Since no $\F_i$ has a first integral, by  Remark \ref{exfield}  the extactics $\boldsymbol{\varepsilon}(\F_i,V)$, $1 \leq i \leq r$, are
not identically zero and given given by
\begin{equation}\label{tkx1}
\boldsymbol{\varepsilon}(\F_i, V) =\det \left[\begin{array}{cccc}
  s_1& s_2 & \cdots & s_{K}\\
  \\
   X_{i}(s_1) &X_{i}(s_2) & \cdots &X_{i}(s_{K})\\
   \\
   X_{i}^2(s_1) &X_{i}^2(s_2) & \cdots &
X_{i}^2(s_{K})
\\
  \vdots & \vdots & \ddots & \vdots\\
X_{i}^{K-1}(s_1) &X_{i}^{K-1}(s_2) & \cdots &
X_{i}^{K-1}(s_{K})
\end{array}
\right]
\end{equation}
where $\{ s_1, \dots, s_K\}$ is a basis of $V$.

Now, an irreducible hypersurface $M \subset \mathbb{P}^n$ is invariant by $\F$ if, and only if, $M$ is invariant by each one of the foliations $\F_i$, $i=1, \dots, r$. Hence, if $M$ has $f=0$ as irreducible defining equation, then $f$ factors each one of extactics above by Proposition \ref{jvpp} that is, $\boldsymbol{\varepsilon}(\F_i, V) = f \, R_i$ where $R_i$ is a polynomial. Let $\mathscr{E}$ be the product
\begin{equation}\label{prodex1}
\mathscr{E}= \boldsymbol{\varepsilon}(\F_1, V) \dots \boldsymbol{\varepsilon}(\F_r, V).
\end{equation}
If $f_1, \dots, f_N$ are irreducible polynomials of degree $\nu$, counting multiplicities, defining the hypersurfaces invariant by all of the $\F_i$s, then
\begin{equation}\label{prodex2}
\mathscr{E}= \underbrace{(f_1 \dots f_N) R_1}_{=\boldsymbol{\varepsilon}(\F_1, V)}  \dots \underbrace{(f_1 \dots f_N) R_r}_{=\boldsymbol{\varepsilon}(\F_r, V)} =(f_1 \dots f_N)^r R_1 \dots R_r.
\end{equation}
We now count degrees.
$$
\deg (f_1 \dots f_N)^r = r \sum\limits_1^N \deg(f_i) = r \nu N.
$$
By (\ref{prodex2})
\begin{equation}\label{prodex3}
\deg(\mathscr{E})= \sum\limits_1^r \deg(\boldsymbol{\varepsilon}(\F_i, V)) \geq r \nu N
\end{equation}
and as
$$
\deg(\boldsymbol{\varepsilon}(\F_i, V)) = K \nu + (d_i -1) {K \choose 2}
$$
we have, by (\ref{prodex3}),
\begin{equation}\label{prodex4}
\deg(\mathscr{E})= K r \nu + \sum\limits_1^r (d_i -1) {K \choose 2} \geq r \nu N.
\end{equation}
This gives
$$
N \leq K +\frac{\sum\limits_1^r (d_i -1)}{r \nu}  {K \choose 2},
$$
and as $\sum\limits_1^r d_i = \deg (\F)$, we are left with
\begin{equation}\label{prodex5}
N \leq K +\frac{\deg (\F)}{r \nu}  {K \choose 2} - \frac{1}{\nu} {K \choose 2}.
\end{equation}
\end{proof}

\section*{Acknowledgments}
This work was partially supported by CNPq, CAPES and FAPEMIG (Brasil). We thank the referee for useful comments and for introducing to us the work \cite{Est}. The third named author is grateful to IMPA for support and hospitality.

\maketitle

\end{document}